\documentclass[12pt,oneside]{amsart}
\usepackage[english]{babel}
\usepackage{amssymb}
\usepackage{amsthm}
\usepackage{mathtools}

\usepackage[foot]{amsaddr}
\usepackage{amsmath}
\usepackage{a4wide}
\usepackage{nccmath}
\usepackage{esvect}
\usepackage{ulem}
\usepackage{enumitem}
\usepackage{hyperref}
\usepackage{cleveref}
\usepackage{subfiles}
\hypersetup{
    colorlinks,
    linkcolor=red,
    citecolor=black,
    urlcolor=blue
}
\usepackage{verbatim}
\usepackage[dvips]{graphicx}
\usepackage{t1enc}
\usepackage{color}
\usepackage[margin=0.5in]{geometry}
\usepackage{graphicx}
\usepackage{color,soul}
\usepackage[parfill]{parskip}   
\usepackage[numbers,sort&compress]{natbib}
\usepackage{tikz}
\usepackage{amsmath}
\usepackage{caption}
\usepackage{subcaption}
\usepackage{mathtools}
\usepackage[section]{placeins}

\usetikzlibrary{arrows,positioning}
\usetikzlibrary{graphs}
\usetikzlibrary{graphs.standard}

\makeatletter
\def\subsubsection{\@startsection{subsubsection}{3}%
  \z@{.5\linespacing\@plus.7\linespacing}{.1\linespacing}%
  {\normalfont\itshape}}
\makeatother

\newtheorem{thm}{Theorem}[section]

\newtheorem{lemma}[thm]{Lemma}

\newtheorem{conjecture}[thm]{Conjecture}

\newtheorem{proposition}[thm]{Proposition}

\newtheorem{problem}[thm]{Problem}

\newtheorem{corollary}[thm]{Corollary}

\newcommand{\ignore}[1]{}

\begin{document}

\author{Rebekah Herrman}
\address[Rebekah Herrman]{Department of Industrial and Systems Engineering, The University of Tennessee, Knoxville, TN}
\email[Rebekah Herrman]{rherrma2@tennessee.edu}

\author{Stephen G. Z. Smith}
\address[Stephen G. Z. Smith]{Huston-Tillotson University}
\email[Stephen G. Z. Smith]{sgsmith@htu.edu}

\title[Extending Grundy domination to $k$-Grundy domination] 
{Extending Grundy domination to $k$-Grundy domination}

\linespread{1.3}
\pagestyle{plain}

\begin{abstract}

The Grundy domination number of a graph $G = (V,E)$ is the length of the longest sequence of unique vertices $S = (v_1, \ldots, v_k)$ satisfying $N[v_i] \setminus \cup_{j=1}^{i-1}N[v_j] \neq \emptyset$ for each $i \in [k]$. Recently, a generalization of this concept called $k$-Grundy domination was introduced. In $k$-Grundy domination, a vertex $v$ can be included in $S$ if it has a neighbor $u$ such that $u$ appears in the closed neighborhood of fewer than $k$ vertices of $S$. In this paper, we determine the $k$-Grundy domination number for some families of graphs, find degree-based bounds for the $k$-$L$-Grundy domination number, and define a relationship between the $k$-$Z$-Grundy domination number and the $k$-forcing number of a graph.
\end{abstract}
\maketitle 

\section{Introduction}

 Grundy domination is a recently introduced variation on dominating sets of graphs \cite{brevsar2014dominating}. For a graph $G = (V,E)$, the goal of Grundy domination is to find a sequence of vertices of $G$ that is a \textit{closed neighborhood sequence}.  A closed neighborhood sequence is a sequence of vertices, $S = (v_1, \ldots, v_k)$, satisfying $N[v_i] \setminus \cup_{j=1}^{i-1}N[v_j] \neq \emptyset$ for each $i \in [k]$, where $N(v) = \{u : uv \in E(G)\}$, $N[v] = N(v) \cup \{v\}$, and $[k] = \{1, 2, ... , k\}$. The length of such a longest sequence is called the \textit{Grundy domination number}, denoted $\gamma_{gr}(G)$.  For $v \in S$, we say $v$ \textit{footprints} any $u \in N[v_i] \setminus \cup_{j=1}^{i-1} N[v_j]$. The unordered set of vertices $\{v_1, ... , v_k\}$ from the sequence $S$ is denoted $\widehat{S}$. 

 Recently, the related concepts of Grundy total ($t$-Grundy) domination, $Z$-Grundy domination, and $L$-Grundy domination have been introduced and studied extensively \cite{brevsar2016total, brevsar2017grundy, brevsar2020grundy, brevsar2021graphs, lin2019zero, nasini2020grundy, bell2021grundy,herrman2022length, herrman2022proof}. These variants are similar to Grundy domination except they require that the sequence $S$ is an \textit{open neighborhood sequence}, \textit{Z-sequence}, or \textit{L-sequence}, respectively. An open neighborhood sequence satisfies $N(v_i) \setminus \cup_{j=1}^{i-1}N(v_j) \neq \emptyset$, a $Z$-sequence satisfies $N(v_i) \setminus \cup_{j=1}^{i-1}N[v_j] \neq \emptyset$, and an $L$-sequence satisfies $N[v_i] \setminus \cup_{j=1}^{i-1}N(v_j) \neq \emptyset$. The Grundy total domination number of graph $G$, or $t$-Grundy domination number, is denoted $\gamma_{gr}^t(G)$, the $Z$-Grundy domination number of $G$ is denoted $\gamma_{gr}^Z(G)$, and the $L$-Grundy domination number of $G$ is denoted $\gamma_{gr}^L(G)$.

In \cite{brevsar2017grundy}, Bre{\v{s}}ar, Bujt{\'a}s, Gologranc, Klav{\v{z}}ar, Ko{\v{s}}mrlj, Patk{\'o}s, Tuza, and Vizer mention generalizing $Z$-Grundy domination to $k$-$Z$-Grundy domination. In this generalization, $G$ is a graph with minimum degree $\delta(G) \geq k$. A sequence $S = (v_1, v_2, ... , v_n)$ where $v_i \in V(G)$ is a $k$-$Z$-sequence if for each $i$ there exists $u_i$ such that $u_i \in N(v_i)$ and $u_i \in N[v_j]$ for fewer than $k$ vertices, $v_j$, in $(v_1, ... , v_{i-1})$. We shall require that no vertex $v_i \in V(G)$ appears in $S$ more than once, i.e., there are no repeated elements in $S$. The length of the longest $k$-$Z$-sequence of $G$ is called the $k$-$Z$-Grundy domination number, denoted $\gamma_{gr}^{Z,k}(G)$. The variants $k$-Grundy, $k$-$t$-Grundy, and $k$-$L$-Grundy domination are defined similarly: $S = (v_1, v_2, ... , v_n)$ is a $k$-sequence if for each $i$ there exists $u_i$ such that $u_i \in N[v_i]$ and $u_i \in N[v_j]$ for fewer than $k$ vertices, $v_j$, in $(v_1, ... , v_{i-1})$, it is a $k$-$L$-sequence if $u_i \in N[v_i]$ and $u_i \in N(v_j)$ for fewer than $k$ vertices, $v_j$, in $(v_1, ... , v_{i-1})$, and  it is a $k$-$t$-sequence if $u_i \in N(v_i)$ and $u_i \in N(v_j)$ for fewer than $k$ vertices, $v_j$, in $(v_1, ... , v_{i-1})$. The $k$-Grundy domination numbers for these variants are denoted $\gamma_{gr}^{k}(G)$, $\gamma_{gr}^{t,k}(G)$, and $\gamma_{gr}^{L,k}(G)$, respectively, and we call their corresponding sequences $k$-sequences, $k$-$t$-sequences, and $k$-$L$-sequences. We say that a vertex $v \in S$ \textit{$m$-$Z$-footprints} a vertex $u$ if $u$ appears in $m-1$ closed neighborhoods of vertices that appear in $S$ before $v$ and $u$ is in the open neighborhood of $v$. The concepts of $m$-footprint, $m$-$t$-footprint, and $m$-$L$-footprint are defined similarly, but with the corresponding open and closed neighborhoods in the definitions of $k$-Grundy, $k$-$t$-Grundy, and $k$-$L$-Grundy domination. When the context is clear, the $L$, $t$, or $Z$ may be omitted when referencing footprinting.

In this paper, we first develop degree-based bounds for $\gamma_{gr}^{L,k}(G)$ and explore inequalities between the different types of $k$-Grundy domination in Section~\ref{sec:degree}. Next, we discuss relationships between $\gamma_{gr}^{Z,k}(G)$ and the $k$-forcing number of $G$ in Section~\ref{sec:kforcing}. We then calculate $\gamma_{gr}^{k}(G)$, $\gamma_{gr}^{t,k}(G)$,  $\gamma_{gr}^{Z,k}(G)$, and $\gamma_{gr}^{L,k}(G)$ in Section~\ref{sec:grundy} for different families of graphs, which relies on results in the previous two sections.  Finally, we close with open problems in Section~\ref{sec:conclusion}.

\section{Degree based bounds}\label{sec:degree}
In this section, we examine bounds for the $k$-$L$-Grundy and $k$-$t$-Grundy domination numbers of graphs based on the minimum degree of the graph.

\begin{thm}\label{thm:degreeL}
Let $G$ be a graph with $n$ vertices and minimum degree $\delta$. Then $\gamma_{gr}^{L,k}(G) \leq n-\delta + k$.
\end{thm}

Note that this is a generalization of a result in \cite{herrman2022length}, and can be proven in the same manner. We restate the proof here in generality for completeness.

\begin{proof}

Let $G$ be an $n$-vertex graph with minimum degree $\delta$. Suppose, for the sake of contradiction, that 
$\gamma_\mathrm{gr}^{L,k}(G)=m \geq n - \delta + k+1$. For convenience, define $t = n - \delta + k+1$ and let 
$S = (v_1, \ldots, v_m)$ be a maximal $k$-$L$-sequence of $G$. Define $T = V(G) \setminus S$. It follows that $|T|= m-n \leq n-t = n-(n - \delta + k+1) = \delta - k-1$.

Since $v_m$ has degree at least $\delta$, and 
there are at most $\delta - k-1$ vertices in $T$, $v_m$ must be adjacent to at least $k+1$ 
vertices in $\widehat{S}$. Thus, $v_m$ can only be in $S$ if there exists $v \in N(v_m)$ that appears in the open neighborhoods of at most $k-1$ vertices in $\widehat{S} \setminus v_m$. If $v$ has at most $k-1$ neighbors in $S$ and is in the open neighborhood of $v_m$, it has at least $\delta - k$ neighbors in $T$, which is not possible.

Therefore, $\gamma_\mathrm{gr}^{L,k}(G) \leq n - \delta + k$.
\end{proof}

The upper bound is tight: for example, $\gamma_{gr}^{L,k} (K_n)=n-(n-1)+k = k+1$, which is proven in the next section. Another example of tightness can be seen in the following graph: Let $T$ and $T'$ both be complete binary trees of the same height, $h \geq 3$, and connect the leaves of $T$ to the leaves of $T'$ in a cycle. We say the roots of each tree are located at level $m=1$ in their respective trees, their neighbors are located at level $m=2$, etc., with the leaves being located at height $m=h$.

This graph has $\delta = 2$, and satisfies $\gamma_\mathrm{gr}^{L,2}(G)= n$. To see this, first add the leaves of each tree into $S$. These vertices can be added into $S$ in any order, since each leaf has a neighbor that has appeared in the open neighborhood of at most one element of $S$ since $h \geq 3$. Then add the neighbors of the leaves. These can be added since their neighbor in $T$ at level $h-2$ has not appeared in an open neighborhood. This process continues until we have added the vertices at level $m=3$ to $S$. At this point, we add the root of each tree to $S$, which can be added since the roots $L$-footprint themselves and they have not appeared in the open neighborhood of any element of $S$. Finally, we add the neighbors of the roots to $S$. These neighbors can be added to $S$ since the root of each tree does not appear in the open neighborhood of any vertex of $S$ (even though they appear in their own closed neighborhoods). An example of this graph is found in Figure~\ref{fig:tightexample}.

\begin{figure}
 \centering
 \begin{tikzpicture}[scale=1]
\begin{scope}[every node/.style={scale=.75,circle,draw}]
    \node[fill=green] (A) at (0,5) {};
    \node (B) at (-1.75,3) {};
	\node (C) at (1.75,3) {}; 
	\node[fill=red] (D) at (-2.25,1) {};
	\node[fill=red] (E) at (-1.25,1) {};
    \node[fill=red] (F) at (2.25,1) {};
	\node[fill=red] (G) at (1.25,1) {}; 
	\node[fill=red] (H) at (1.25,-1) {};
	\node[fill=red] (I) at (2.25,-1) {};
    \node[fill=red] (J) at (-1.25,-1) {};
	\node[fill=red] (K) at (-2.25,-1) {}; 
	\node (L) at (1.75,-3) {};
	\node (M) at (-1.75,-3) {}; 
	\node[fill=green] (N) at (0,-5) {};

\end{scope}

\draw  (A) -- (B);
\draw  (A) -- (C);
\draw  (B) -- (D);
\draw  (B) -- (E);
\draw  (C) -- (F);
\draw  (C) -- (G);

\draw  (D) -- (K);
\draw  (D) -- (J);
\draw  (E) -- (J);
\draw  (E) -- (H);
\draw  (F) -- (I);
\draw  (G) -- (I);
\draw  (G) -- (H);
\draw  (F) -- (K);

\draw  (L) -- (H);
\draw  (L) -- (I);
\draw  (M) -- (J);
\draw  (M) -- (K);
\draw  (N) -- (L);
\draw  (N) -- (M);

\end{tikzpicture}
\caption{An example of a graph for which equality in Theorem~\ref{thm:degreeL} holds. All red-filled vertices are added to $S$ first in any order, followed by the green-filled vertices, then the unfilled ones.}
\label{fig:tightexample}
\end{figure}
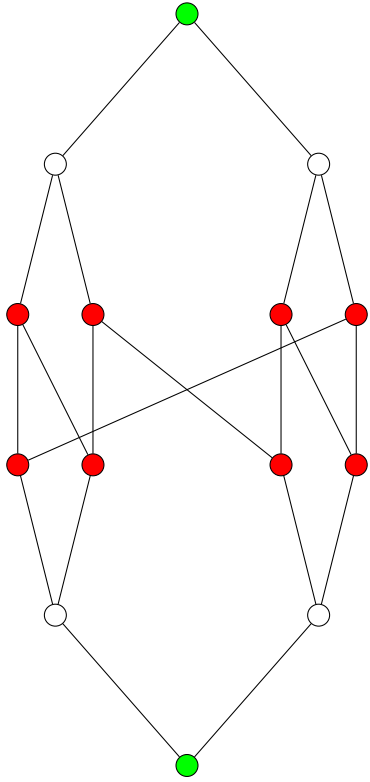

Since $\delta(G) \geq k$ is a condition for $k$-$Z$-Grundy domination, Theorem~\ref{thm:degreeL} implies that in order for $\gamma_{gr}^{L,k}(G) = n$, $\delta = k$. This condition is necessary but may not be sufficient. 

The following Proposition gives a comparison between $\gamma_{gr}^{k}(G)$ and $\gamma_{gr}^j(G)$ for $k \neq j$.

\begin{proposition}
For a graph $G$, $\gamma_{gr}^{k}(G) \leq \gamma_{gr}^{j}(G)$ and $\gamma_{gr}^{a,k}(G) \leq \gamma_{gr}^{a,j}(G)$ for $a \in \{L,Z,t\}$ and $k < j$.
\end{proposition}

This clearly must hold since if $v$ is not in the neighborhood of $k$ vertices in $S$, it cannot be in the neighborhood of $j$ vertices of $S$ when $k < j$. Thus, if the $k$-Grundy domination number of a graph can be determined, it is a lower bound for the $j$-Grundy domination number of the same graph, where $k < j$.

Furthermore, we can compare $\gamma_{gr}^{k}(G)$, $\gamma_{gr}^{L,k}(G)$, $\gamma_{gr}^{Z,k}(G)$, and $\gamma_{gr}^{t,k}(G)$.

\begin{proposition}\label{prop:comparison}
$\gamma_{gr}^{Z,k}(G) \leq \gamma_{gr}^k(G) \leq \gamma_{gr}^{L,k}(G)-1$ and $\gamma_{gr}^{Z,k}(G) \leq \gamma_{gr}^{t,k}(G) \leq \gamma_{gr}^{L,k}(G)$.
\end{proposition}

The proof of this result is the same as the proof of the analogous result for $k=1$ in \cite{brevsar2017grundy}. Furthermore, these inequalities are tight. To see this, $\gamma_{gr}^{t,k}(K_n) = \gamma_{gr}^{L,k}(K_n)$ for all $n \geq k+1$ and $\gamma_{gr}^k(K_n) = \gamma_{gr}^{L,k}(K_n)-1$ for the same family of graphs. Additionally, $\gamma_{gr}^{Z,k}(K_{m,n}) = \gamma_{gr}^{t,k}(K_{m,n})$ for $m> k, n \geq k$ and $\gamma_{gr}^{Z,k}(K_{m,n}) = \gamma_{gr}^k(K_{m,n})$ when $m = n = k$. These results are proven in the next section.

The degree-based upper bound for the other types of Grundy domination immediately follows from Theorem \ref{thm:degreeL} and Proposition~\ref{prop:comparison}.

\begin{corollary}\label{cor:degreebound}
$\gamma_\mathrm{gr}^{t,k}(G) \leq n - \delta + k$ and $\gamma_\mathrm{gr}^{Z,k}(G) \leq \gamma_\mathrm{gr}^{k}(G) \leq n - \delta + k-1$.
\end{corollary}

\section{Relationship to $k$-forcing}\label{sec:kforcing}

The $k$-forcing number of a graph is a generalization of the zero forcing number. In zero forcing, a collection of vertices are colored blue, and the rest are white. White vertices can be changed blue at discrete time steps according to the \textit{color change rule}. The color change rule states that if a blue vertex, $b$, has exactly one neighbor that is white, $w$, then $w$ is turned blue, and we say that $b$ \textit{forces} $w$. The minimum number of blue vertices needed to ensure the entire graph is colored blue eventually is called the \textit{zero forcing number} of $G$, denoted $Z(G)$. Zero forcing can be used to determine the minimum rank and maximum nullity of graphs \cite{AIMMINIMUMRANKSPECIALGRAPHSWORKGROUP20081628, barioli2010zero, huang2010minimum, edholm2012vertex}, and is closely related to power domination \cite{benson2018zero, bozeman2019restricted}.

The generalization of zero forcing is called $k$-forcing. In $k$-forcing, a subset of vertices is colored blue while the rest are white. The color change rule here differs slightly from zero forcing in that if a blue vertex has at most $k$ white neighbors, all white neighbors are colored blue at the next time step. The $k$-forcing number of a graph $G$ is denoted $F_k(G)$ and is the size of the smallest $k$-forcing set of the graph. Upper bounds of $k$-forcing numbers of graphs and its relationship to $k$-power domination has been studied in recent years \cite{caro2014dynamic, amos2015upper, ferrero2018relationship}. In \cite{brevsar2017grundy}, Bre{\v{s}}ar et al. prove that $\gamma_{gr}^Z(G) + Z(G) = |V(G)|$, however the relationship between $k$-$Z$-Grundy domination and $k$-forcing cannot be established as easily. In fact,

\begin{thm}\label{prop:relationtokforcing}
Let $G = (V,E)$ with $|V(G)| = n$. Then,
$\gamma_{gr}^{Z,k}(G) \geq n-F_k(G)$.
\end{thm}

\begin{proof}
The proof of this result for general $k$ is analogous to the proof when $k=1$, which was originally proven in \cite{brevsar2017grundy}, however there are some details that must be slightly modified. We write the proof here for general $k$ for completeness. 

Without loss of generality, let $G$ be a connected graph and let $B$ be a $k$-forcing set. Let $m = n - |B|$. Consider the sequence $\{b_1, ... , b_t\}$ where each $b_i$ is a blue vertex that colors at least one white vertex at step $i$ of the color change process. Let $\{w_{i}\}$ be the at most $k$ white neighbors of $b_i$ at the moment it is chosen. Suppose there are $t$ steps of the color change process. 
We will show that $(\{w_{t}\}, ... , \{w_{1}\})$ is a valid Z-sequence, where vertices in each $\{w_{i}\}$ can be selected in any order as long as all vertices of $\{w_{i}\}$ are selected before any vertex of $\{w_{i-1}\}$ is.

To see this, $b_i \notin\{w_{i}\}$ and it is contained in the open neighborhoods of all vertices in $\{w_i\}$ where $|\{w_i\}| \leq k$. Let $d< t$. Then $N(b_d) \subset V(G) \setminus (\{w_{d+1}\} \cup \ldots \cup \{w_t\})$ otherwise 
$b_d$ cannot force any vertices because it now has more than $k$ white neighbors, or if $v \in N(b_d) \cap \{w_c\}$ for some $c \geq d+1$ and $|N(b_d)| < k$, then $v$ would be placed in $\{w_d\}$ initially and not $\{w_c\}$.  
Thus each vertex in $\{w_{d}\}$ can be placed in $S$ because $b_d$ will have been in the open neighborhood of at most $k$ vertices of $S$. 
If $B$ is a minimal $k$-forcing set, then $ F_k(G)=|B| = n-m $ and $m \leq \gamma_{gr}^{Z,k}(G)$, so $\gamma_{gr}^{Z,k}(G) \geq n-F_k(G)$.

\end{proof}

We believe the following conjecture is true.
\begin{conjecture}
Let $G = (V,E)$ with $|V(G)| = n$. Then,
$\gamma_{gr}^{Z,k}(G) = n-F_k(G)$.
\end{conjecture}
\noindent The argument used in \cite{brevsar2017grundy} is not sufficient to prove equality. In \cite{brevsar2017grundy}, the authors define a sequence of vertices $A=\{a_1, ... , a_m\}$, where $a_i$ is footprinted by $u_i \in S$. They then show that $a_m$ must be blue, and $a_i$ forces $u_i$ for all $1 \leq i \leq m$. When attempting this same argument for $k$-forcing, it is not clear that $u_m$ has a blue neighbor.





\section{The k-Grundy domination number for different classes of graphs}\label{sec:grundy}

In order to gain intuition as to how $k$-Grundy, $k$-$L$-Grundy, $k$-$Z$-Grundy, and $k$-$t$-Grundy domination numbers are related, we first calculate the aforementioned quantities for specific families of graphs. 






\subsection{Cycles}\label{subsec:cycle}
A cycle on $n$ vertices, $C_n$, is a connected graph in which all vertices have degree two. Since $\gamma_{gr}(C_n)$ and $\gamma_{gr}^{t}(C_n)$ have been found previously \cite{brevsar2016total}, we will determine the different $2$-Grundy domination numbers for $C_n$. 
\begin{thm}

\begin{enumerate}
\renewcommand\labelenumi{\normalfont(\arabic{enumi})}
\item $\gamma_{gr}^{2}(C_n)= n-1$
\item $\gamma_{gr}^{L,2}(C_n)= n$
\item $\gamma_{gr}^{Z,2}(C_n)= n-1$
\item $\gamma_{gr}^{t,2}(C_n)= n$.
\end{enumerate}

\end{thm}

\begin{proof}

Select a vertex in $C_n$ and label it $v_1$. Label vertices clockwise from this point in increasing order.  

(1) First, $\gamma_{gr}^{2}(C_n) \leq n-1$ follows from Corollary~\ref{cor:degreebound} with $\delta = k = 2$. 
We now show that there exists a 2-Grundy sequence of length $n-1$. Let $S = \{v_1, v_2, ... , v_{n-1}\}$. To see that this is a 2-Grundy sequence, note that each vertex $v_i$ for $1 \leq i \leq n-2$ has a neighbor $v_{i+1}$ that has not been in $N[v_j]$ for $j < i$. This takes care of the first $n-2$ elements of $S$. Now, $v_1$ is the only neighbor of $v_n$ in $\{v_1, \ldots, v_{n-2}\}$, so $v_{n-1}$ can be added to $S$. Thus, $\gamma_{gr}^{2}(C_n)= n-1$.

(2) By Theorem~\ref{thm:degreeL}, $\gamma_{gr}^{L,2}(C_n) \leq n$. To see that $S = \{v_1, v_2, ..., v_n\}$ is a $2$-$L$-sequence, note that $v_j$ for $j \in [n-2]$ $1$-$L$-footprints $v_{j+1}$. Thus all $v_j$ for $j \in [n-2]$ can be included in $S$. The only vertex in $S$ containing $v_{n-2}$ in its open neighborhood is $v_{n-3}$, so $v_{n-1}$ $2$-$L$-footprints $v_{n-2}$ and can thus be included in $S$. Similarly, $v_n$ $2$-$L$-footprints $v_1$. Thus, $\gamma_{gr}^{L,2}(C_n)= n$.

(3)  By Corrollary~\ref{cor:degreebound}, $\gamma_{gr}^{Z,2}(C_n) \leq n-1$. The proof of $\gamma_{gr}^{Z,2}(C_n) \geq n-1$ follows from the proof of (1) since $N(v_1) \subset N[v_1]$.  Thus, $\gamma_{gr}^{Z,2}(C_n)= n-1$.

(4) Let $S = (v_1, v_2, ..., v_n)$. This is a $2$-$t$-sequence because for each $j \in [n-2]$, $v_{j+1} \notin N(v_i)$ for $i < j$. In order to include $v_{n-1}$ in $S$, note that out of all the vertices in $S$ before $v_{n-1}$, $v_{n-2}$ appears only in the open neighborhood of $v_{n-3}$. Thus, $v_{n-1}$ $2$-footprints $v_{n-2}$. Finally, note that the only vertex in $\{v_1, ... , v_{n-1}\}$ that contains $v_1$ in its open neighborhood is $v_2$, therefore $v_n$ can be included in $S$ since it $2$-footprints $v_1$. Thus, $\gamma_{gr}^{t,2}(C_n)= n$.
\end{proof}

\subsection{Complete graphs}\label{subsec:complete}
The complete graph on $n$ vertices, denoted $K_n$, is the graph in which every vertex is adjacent to all other vertices. Thus, each vertex has degree $n-1$. The following result holds for $K_n$.
\begin{thm}
 For $k \leq n-1$,
\begin{enumerate}
\renewcommand\labelenumi{\normalfont(\arabic{enumi})}
\item $\gamma_{gr}^{k}(K_n)= k$
\item $\gamma_{gr}^{L,k}(K_n)= k+1$
\item $\gamma_{gr}^{Z,k}(K_n)= k$
\item $\gamma_{gr}^{t,k}(K_n)= k+1$.
\end{enumerate}
\end{thm}

\begin{proof}

(1) By Corollary~\ref{cor:degreebound}, $\gamma_{gr}^{k}(K_n)\leq n-\delta + k-1 = n-(n-1)+k-1 = k$. For the reverse inequality, note that for distinct $v_i, v_j \in V(K_n)$, $v_i \in N(v_j)$, and furthermore, $v_i \in N[v_i]$. Thus, in order for an arbitrary $v_i$ to be $k$-footprinted, there must be at least $k$ vertices in $S$. Therefore $\gamma_{gr}^{k}(K_n)= k$. 

(2) By Theorem~\ref{thm:degreeL}, $\gamma_{gr}^{L,k}(K_n) \leq k+1$ so it remains to show $\gamma_{gr}^{L,k}(K_n) \geq k+1$. Label the vertices of $K_n$ $v_1$ to $v_n$. For all $i \neq j$, $v_i \in N(v_j)$. Thus, any vertex in $S$ is in the open neighborhood of $|S|-1$ vertices of $S$ and every vertex not in $S$ is in the open neighborhood of all vertices of $S$. When $|S| = k+1$, all vertices have been included in the open neighborhood of at least $k$ vertices of $S$. Thus, $\gamma_{gr}^{L,k}(K_n) \geq k+1$. Since $\gamma_{gr}^{L,k}(K_n) \leq k+1$ and $\gamma_{gr}^{L,k}(K_n) \geq k+1$, $\gamma_{gr}^{L,k}(K_n)= k+1$.

(3) This proof is analogous to the proof of (1) with the appropriate open neighborhoods.

(4) This proof is analogous to the proof of (2) with the appropriate open neighborhoods.
\end{proof}

\subsection{Complete bipartite graphs}\label{subsec:completebipartite}
We now consider complete bipartite graphs. A complete bipartite graph is one in which the vertices are partitioned into two sets, $M$ and $N$. All $m$ vertices in $M$ are adjacent to all $n$ vertices in $N$, no two vertices in $M$ are adjacent and no two vertices in $N$ are adjacent. This graph is denoted $K_{m,n}$, and throughout this subsection, we assume $m \geq n$.
\begin{thm}
For $m,n \geq k$ and $m \geq n$,
\begin{enumerate}
\renewcommand\labelenumi{\normalfont(\arabic{enumi})}
\item $\gamma_{gr}^{k}(K_{m,n}) = m + k-1$
\item $\gamma_{gr}^{L,k}(K_{m,n}) = m + k$
\item $\gamma_{gr}^{Z,k}(K_{m,n})=
 \begin{cases}
                    2k & \; if \; m> k, n \geq k  \\
                    2k-1& \; if \; m,n = k
\end{cases}$
\item $\gamma_{gr}^{t,k}(K_{m,n})= 2k$.
\end{enumerate}
\end{thm}

\begin{proof}

(1) First, we shall show $\gamma_{gr}^{k}(K_{m,n}) \geq m + k-1$ and then show $\gamma_{gr}^{k}(K_{m,n}) \leq m + k-1$. For $v \in M$, $N[v] = N \cup v$, and for $u \in N$, $N[u] = M \cup u$. All $v \in M$ can be added to $S$ since for all $u \in M$, where $u \neq v$, $u \notin N[v]$. Now, each element of $N$ has been in the closed neighborhoods of $m \geq k$ elements of $S$. Once these are in $S$, $k-1$ elements of $N$ can be added to $S$ since each element of $M$ is in the closed neighborhood of each element of $N$ and each element of $M$ is in the closed neighborhood of exactly one element of $S$ prior to adding elements of $N$ to $S$. Thus, $\gamma_{gr}^{k}(K_{m,n}) \geq m + k-1$. 

By Corollary~\ref{cor:degreebound}, $\gamma_{gr}^{k}(K_{m,n}) \leq m+n-\delta + k-1 = m+n-n+k-1 = m+k-1$.

(2)  $\gamma_{gr}^{L,k}(K_{m,n}) \geq m + k$ is true by Proposition~\ref{prop:comparison} and (1) while $\gamma_{gr}^{L,k}(K_{m,n}) \leq m + k$ holds by Theorem~\ref{thm:degreeL}.


(3) First, we consider the case where $m = n = k$. $\gamma_{gr}^{Z,k}(K_{m,n}) \leq 2k-1$ by Corollary~\ref{cor:degreebound}, so it remains to show $\gamma_{gr}^{Z,k}(K_{m,n}) \geq 2k-1$.  For $u \in M$, $N(u) = N$ so all $k$ elements of $M$ can be added to $S$ since no element of $M$ can $Z$-footprint another element of $M$. At this point, each element of $M$ has appeared in the closed neighborhood of one element of $S$, and each element of $N$ has appeared in the open neighborhood of $k$ elements of $S$. Thus, when adding elements from $N$ to $S$, only $k-1$ vertices can be included since each element of $M$ is in the neighborhood of each element of $N$, so $\gamma_{gr}^{Z,k}(K_{m,n}) \geq 2k-1$ when $m=n=k$. Thus, $\gamma_{gr}^{Z,k}(K_{k,k}) = 2k-1$.

We now consider $ m> k, n \geq k$, and must show $\gamma_{gr}^{Z,k}(K_{m,n}) \geq 2k$ and $\gamma_{gr}^{Z,k}(K_{m,n}) \leq 2k$. For the former, we can add $k$ elements of $M$ to $S$, since each $Z$-footprints all elements of $N$, and each element of $N$ can be $Z$-footprinted at most $k$ times. Now, there exists at least one element of $M$, say $v$, that has not appeared in the closed neighborhood of any element in $S$. Thus, we can add $k$ elements of $N$ to $S$, since each element of $N$ $Z$-footprints $v$.  Thus, $\gamma_{gr}^{Z,k}(K_{m,n}) \geq 2k$. For the latter, by Proposition~\ref{prop:comparison}, $\gamma_{gr}^{Z,k}(K_{m,n})\leq \gamma_{gr}^{t,k}(K_{m,n})$, and $\gamma_{gr}^{t,k}(K_{m,n}) = 2k$ is proven next.

(4) We shall show $\gamma_{gr}^{t,k}(K_{m,n}) \leq 2k$ and $\gamma_{gr}^{t,k}(K_{m,n}) \geq 2k$. For the former, suppose for the sake of contradiction that $\gamma_{gr}^{t,k}(K_{m,n}) > 2k$. Then at least $2k+1$ vertices from either $M$ or $N$ is in $S$. Without loss of generality, suppose there are at least $2k+1$ vertices in $M$ included in $S$. Consider the last vertex of $M$ in $S$, say $v$. Now, $v$ cannot $k$-$t$-footprint itself by definition, so it must $k$-$t$-footprint some vertex in $N$. This is impossible, however, since every vertex in $N$ is in the open neighborhood of at least $k$ vertices that appear in $S$ before $v$. Therefore $\gamma_{gr}^{t,k}(K_{m,n}) \leq 2k$. To show the reverse inequality, note that the degree of all vertices in $M$ is $n$ and the degree of all vertices in $N$ is $m$. For each $v_i \in N$, $N(v_i) = M$. Thus, at most $k$ elements of $M$ can be added to $S$ and, by symmetry, at most $k$ elements of $N$ can be added to $S$. It is easy to see that exactly $k$ vertices from $M$ and $k$ from $N$ can be added to $S$, so $\gamma_{gr}^{t,k}(K_{m,n})\geq 2k$. Combining the two inequalities gives the result.

\end{proof}

\subsection{$d$-dimensional discrete hypercubes}\label{subsec:cube}
 The $d$-dimensional discrete hypercube, $Q_d$, has a vertex set that consists of all $0-1$ valued sequences of length $d$, $\epsilon =(\epsilon_i)_{i=1}^d$, where $\epsilon_i \in \{0,1\}$. Two sequences are joined by an edge if they differ in exactly one place. 
 
 In \cite{brevsar2017grundy}, it was shown that the $Z$-Grundy domination number of the $d$-dimensional cube, $Q_d$, is $\gamma_{gr}^Z(Q_d) = 2^{d-1}$. We shall show that the $k$-$L$-Grundy number of $Q_d$ is at least $ \lceil 2^d - 2^{d-(k+1)} \rceil$. 

The main result in this subsection relies on adding a specific pattern of vertices to $S$ in order to construct lower bounds for the $k$-L-Grundy domination number of $Q_d$. We shall define that pattern now. Consider $Q_d$ with $d \geq 3$. Partition $Q_d$ into $3$-dimensional discrete hypercubes, $\{Q_i\}$ such that the vertices of $Q_i$ consist of sequences satisfying the condition that the first $d-3$ positions of the sequences are identical, and the last three positions range from $000, 001, ... , 111$. Identify the cube $Q_1$ with sequences such that the first $d-3$ entries are $0$, $Q_2$ with all sequences where the first $d-3$ entries are $0$ except for the first entry, $Q_3$ with all sequences where the first $d-3$ entries are $0$ except for the second entry, etc, until $Q_{2^{d-3}}$ is identified with the sequence that has ones in the first $d-3$ positions. We now define \textit{Pattern A} $=\{*000,*011,*101,*110\}$ and \textit{Pattern B} $=\{*001,*010,*100,*111\}$, where  $* \in \{0,1\}^{d-3}$. Since any pair of vertices from the same pattern have Hamming distance greater than 1, they are not neighbors. Define the \textit{cube distance} between two cubes $Q_i$ and $Q_j$ as the Hamming distance between the vertex of $Q_i$ that has last three entries 000 and the vertex of $Q_j$ with last three entries 000. The \textit{standard pattern} for $Q_d$ consists of all Pattern A vertices of cubes that have even cube distance to $Q_1$ and all Pattern B vertices of cubes that have odd cube distance to $Q_1$. 

\begin{lemma}
The vertices in the standard pattern added to $S$ in any order form a $L$-sequence.
\end{lemma}

\begin{proof}
No vertices in the standard pattern are neighbors by construction, so they always footprint themselves.
\end{proof}

We now find a lower bound for $\gamma_{gr}^{L,k}(Q_d)$ for arbitrary $k$ and $d \geq 2$.

\begin{thm}\label{thm:Lcube}
For $d \geq 2$, and $1 \leq k \leq d$, $\gamma_{gr}^{L,k}(Q_d) \geq \lceil 2^d - 2^{d-(k+1)} \rceil$
\end{thm}

\begin{proof}
 Note that $Q_2 = C_4$, $\gamma_{gr}^L(C_4) = 3 = 2^2-2^0$, and $\gamma_{gr}^{L,2}(C_4) = 4 = \lceil 2^2 - 2^{2-3} \rceil$. For the remainder of this proof, we assume $d \geq 3$.

Let us fix $k \leq d$. Throughout this section, we define $S^C = Q_d\setminus \widehat{S}$. Add vertices of $Q_d$ to $S$ in the standard pattern. Note all vertices in $S$ have not been footprinted and all vertices in $S^C$ have been footprinted $d$ times. Thus, any vertex that will be added to $S$ in the future must footprint a vertex in $Q_d \cap S$. 

Split $Q_d$ into two $(d-1)$-dimensional hypercubes, $Q$ and $Q'$, where the first entry of every vertex in $Q$ is 0 and the first entry of every vertex in $Q'$ is 1. Each vertex in $Q \cap S^C$ $1$-$L$-footprints exactly one vertex in $Q' \cap S$, since each vertex in $Q' \cap S$ has exactly one neighbor in $Q \cap S^C$ by construction. Thus, the vertices in $Q \cap S^C$ can be added to $S$. If $k=1$, every vertex has been 1-footprinted, so no more vertices can be added to $S$ and $|S| = \lceil 2^d - 2^{d-2} \rceil $.

If $k > 1$, more vertices can be added to $S$ since the vertices of $Q' \cap S$ have only been $1$-$L$-footprinted. To determine which vertices to add to $S$, split $Q'$ into two $(d-2)$-dimensional hypercubes, $Q''$ and $Q'''$, where the second entry of vertices in $Q''$ is 0 and the second entry of vertices on $Q'''$ is 1. Each vertex in $Q'' \cap S^C$ $2$-$L$-footprints exactly one vertex in $S \cap Q'''$, since each vertex in $S \cap Q'''$ has been $1$-$L$-footprinted by a vertex in $Q'$ and each vertex in $S \cap Q'''$ has exactly one neighbor in $Q'' \cap S^C$ by construction. Thus, each vertex in $Q'' \cap S^C$ can be added to $S$. If $k=2$, every vertex has been $2$-$L$-footprinted, so no more vertices can be added to $S$ and $|S| = \lceil 2^d - 2^{d-3} \rceil $.

This process continues until we split some $(d-(k-1))$-dimensional hypercube into two $(d-k)$-dimensional hypercubes (for $k < d$), $Q^a$ and $Q^b$, where the $k^{th}$ entry of vertices in $Q^a$ is 0 and the $k^{th}$ entry of vertices on $Q^b$ is 1. All vertices in $Q_a \cap S$ and $Q_b \cap S$ have been $(k-1)$-$L$-footprinted at this point, and each vertex in $Q^a \cap S^C$ $k$-$L$-footprints exactly one vertex in $S \cap Q^b$. Thus, each vertex in $Q^a \cap S^C$ can be added to $S$, each vertex in $Q_d$ has been $k$-$L$-footprinted, and $|S| = \lceil 2^d - 2^{d-(k+1)} \rceil$ when $k < d$. 

When $k=d$, there is only one vertex not in $S$. This vertex has $d$ neighbors, each with degree $d$. Thus, these neighbors have been $(d-1)$-footprinted, and the last vertex can be added to $S$ since it $d$-footprints them. The ceiling function is only needed in the case when $k=d$, since it is the only non-integer value of $2^d - 2^{d-(k+1)}$. Note that when $k=d$, $|S| = \lceil 2^d - 2^{-1} \rceil = 2^d$. In this case, since $\gamma_{gr}^{L,k}(Q_d)$ is bounded above by $2^d$, equality holds, so $\gamma_{gr}^{L,k}(Q_d) = 2^d$.

Thus, $\lceil 2^d - 2^{d-(k+1)} \rceil$ is a lower bound for $\gamma_{gr}^{L,k}(Q_d)$ since there exists a k-L-sequence of this length.


\end{proof}

Equality holds when $k \in \{d-1, d-2\}$ via the following lemma.

\begin{lemma}\label{dregular}
Let $G$ be a $d$-regular graph and let $k < d$. If $S$ is a maximal $L$-Grundy dominating sequence and $x$ is the last vertex in $S$, then $x$ must footprint another vertex $y \neq x$ for the $k^{th}$ time.
\end{lemma}

\begin{proof}
Suppose by contradiction that $x$ does not footprint a vertex for the $k^{th}$ time, but rather is in $S$ only because it is contained in fewer than $k$ neighborhoods of vertices of $S$. Since $G$ is $d$-regular, not all of the neighbors of $x$ are in $S$. Let $y$ be one such neighbor. We can extend $S$ by adding $y$ to the end, since $x$ is in fewer than $k$ open neighborhoods of elements of $S$, contradicting the maximality of $S$.
\end{proof}

\begin{proposition}
For $d > 2$, $\gamma_{gr}^{L,d-1}(Q_d) = \lceil 2^d - 2^{d-(d-1+1)} \rceil = 2^d-1$ and $\gamma_{gr}^{L,d-2}(Q_d) = \lceil 2^d - 2^{d-(d-2+1)} \rceil = 2^d-2$.
\end{proposition}
\begin{proof}
We need to show $\gamma_{gr}^{L,d-1}(Q_d) \leq \lceil 2^d - 2^{d-(d-1+1)} \rceil = 2^d-1$ and $\gamma_{gr}^{L,d-2}(Q_d) \leq \lceil 2^d - 2^{d-(d-2+1)} \rceil = 2^d-2$, since the reverse directions of both equalities was proven in Theorem~\ref{thm:Lcube}. Consider $k=d-1$, and let $S$ be a maximal $(d-1)$-$L$-Grundy dominating sequence. There are at least $2^d - 2^{d-(k+1)} = 2^d - 2^{d-(d-1+1)} = 2^d-1 $ vertices in $S$, so at most one vertex is not in $S$. Suppose for the sake of contradiction that $\gamma_{gr}^{L,d-1}(Q_d) = 2^d$, so all vertices are in $S$. Let $v_m$ be the last vertex in $S$. By Proposition~\ref{dregular}, $v_m$ must footprint some $v_i$ for $i \neq m$ for the $(d-1)^{st}$ time. This is not possible since all vertices have degree $d$, so all have been $L$-footprinted at least $d-1$ times before $v_m$ is added to $S$, a contradiction.

When $k=d-2$, there are at least $2^d - 2^{d-(k+1)} = 2^d - 2^{d-(d-2+1)} = 2^d-2 $ in $S$, so at most two vertices are not in $S$. Suppose for the sake of contradiction that all $\gamma_{gr}^{L,d-1}(Q_d) \geq 2^d-1$. Then all vertices are in $S$, except possibly one that we shall call $x$. Let $v_m$ be the last vertex in $S$. By Proposition~\ref{dregular}, $v_m$ must footprint some $v_i$ for $i \neq m$ for the $(d-2)^{nd}$ time or must footprint $x$ for the $(d-2)^{nd}$ time. Since $x$ has degree $d$, at least $d-2$ neighbors of $x$ are in $S$ so $v_m$ cannot $(d-2)$-footprint $x$. Similarly, $v_m$ cannot footprint $v_i$ for the $(d-2)^{nd}$ time since all $v_i$ have degree $d$ and have at least $d-2$ neighbors in $S$ before $v_m$, a contradiction.

\end{proof}
A simple upper bound for $\gamma_{gr}^{L,k}(Q_d)$ follows from Theorem~\ref{thm:degreeL}.

\begin{thm}
$\gamma_{gr}^{L,k}(Q_d) \leq 2^d - d + k$.
\end{thm}
After calculating $\gamma_{gr}^{L,k}(Q_d)$ for small values of $k$ and $d$, we conjecture that equality holds for all $k$:

\begin{conjecture}\label{cubeconjecture}
$\gamma_{gr}^{L,k}(Q_d) = \lceil 2^d - 2^{d-(k+1)} \rceil$.
\end{conjecture}

\subsection{Grids}\label{subsec:grids}
The grid graph is the graph Cartesian product of two paths, $P_n$ and $P_m$. Let $G = (V,E)$ and $G' = (V',E')$. The Cartesian product of $G$ and $G'$,  denoted $G \Box G' = (V_p, G_p)$, is a graph with vertex set $V_p = \{(u,v) : u \in V, v \in V'\}$ and two vertices $(u,v), (x,y) \in V_p$ are adjacent if $u=x$ and $vy \in E'$ or $v=y$ and $ux \in E$. Bre{\v{s}}ar et al. determined the Grundy domination number of $P_m \Box P_n$ in \cite{brevsar2016dominating}. For finite $m$ and $n$, the smallest degree of a vertex in $P_m \Box P_n$ is two, so we consider only cases when $k =2$. 

\begin{thm}\label{thm:grid}
Let $m,n$ be finite and $m \geq 2, n \geq 1$. Then $\gamma_{gr}^{2}(P_m \Box P_n) = mn-1$.
\end{thm}

\begin{proof}

First, we shall show that $\gamma_{gr}^{2}(P_m \Box P_n) < mn$, and then show $\gamma_{gr}^{2}(P_m \Box P_n) \geq mn-1$.

Suppose that there are $mn-1$ vertices in $S$ and let $v$ be the vertex of $P_m \Box P_n$ that is not in $S$. We shall show that $v$ and its neighbors have all been 2-footprinted, and thus $v$ cannot be added to $S$. First note that since $v$ is not in $S$, all of its neighbors must be in $S$, so $v$ has been 2-footprinted. Let $u$ be a neighbor of $v$. Now, since $m \geq 2$, $u$ has at least one neighbor besides $v$, say $w$. Then $w$ and $u$ footprint $u$, so it has been 2-footprinted. Since $u$ was arbitrary, $v$ cannot 2-footprint any of its neighbors, and hence $v$ cannot be included in $S$.

In order to show that  $\gamma_{gr}^{2}(P_m \Box P_n) \geq mn-1$, we will find a Grundy sequence of length $mn-1$. Orient the graph so that there are $m$ rows and $n$ columns. We refer to each vertex in the graph as $(a,b)$ where $a$ corresponds to the row in which the vertex is located and $b$ the column.  Add all $m$ vertices $(i,1)$ for $1 \leq i \leq m$ vertices to $S$ in any order. These vertices have a neighbor $(i,2)$ such that $(i,2) \notin N[(j,1)]$ when $i \neq j$. Thus, for fixed $i$, the vertex $(i,1)$ 1-footprints the vertex $(i,2)$. Similarly, all vertices $(i,2)$ through $(i,n-1)$ for $i$ ranging from $1$ to $m$ can be included in $S$, so long as all vertices in column $j$ are included in $S$ before any vertex from column $j+1$ is. Every vertex in $S$ up to this point has been footprinted by at least two other vertices of $S$, since each vertex in $S$ has at least two neighbors in $S$.  Now, all vertices $(i,n)$ have now been in the closed neighborhood of exactly one element of $S$ each, namely $(i-1,n)$, so more vertices can be included in $S$. The vertex $(1,n)$ can be added to $S$ since it 2-footprints $(2,n)$. Similarly, the $(2,n)$ can be added to $S$ since it footprints $(3,n)$, and so on, until only $(m,n)$ has not been added to $S$. This vertex cannot appear in $S$ since its only two neighbors have been previously 2-footprinted. Hence, $\gamma_{gr}^{2}(P_m \Box P_n) \geq mn-1$.

Since $\gamma_{gr}^{2}(P_m \Box P_n) < mn$ and $\gamma_{gr}^{2}(P_m \Box P_n) \geq mn-1$,$\gamma_{gr}^{2}(P_m \Box P_n) = mn-1$.
\end{proof}

The $2$-$L$-Grundy domination number follows from Theorem~\ref{thm:grid} and Proposition~\ref{prop:comparison}.
\begin{corollary}\label{lgrid}
Let $m \leq n$ and $n$ finite. Then $\gamma_{gr}^{L,2}(P_m \Box P_n)=mn$.
\end{corollary}

\section{Conclusion}\label{sec:conclusion}


First, we determined the $k$-Grundy domination numbers for different families of graphs, including $P_m \Box P_n$. In $P_m \Box P_n$, $\delta = 2$ so we only calculated $\gamma_{gr}^{2}(P_m \Box P_n)$. For infinite grids, however, the minimum degree is greater than two, so larger $k$ can be considered. Since infinite grids do not have a finite number of vertices, there is no reason to believe that $\gamma_{gr}^{k}$ is finite for these graphs. Let us define $\delta_{gr}^{k}(G)$ to be the density of vertices in a graph $G$ that form a $k$-Grundy dominating set of $G$. It would be interesting to determine the density of vertices that can be included in a $k$-Grundy dominating set for infinite rectangular grids.

\begin{problem}
Determine $\delta_{gr}^{k}(G)$ when $G$ is an infinite grid.
\end{problem}

We believe this number should be close to 1 when $k \leq 4$ due to Theorem~\ref{thm:grid}. It would also be of interest to determine $\gamma_{gr}^k(G)$ where $G$ is a finite regular lattice that is not $P_m \Box P_n$, for example, a triangular lattice.

\begin{problem}
Determine $\gamma_{gr}^{L,k}(G)$ when $G$ is a finite lattice that is not $P_m \Box P_n$.
\end{problem}

We then proved that $\gamma_\mathrm{gr}^{L,k}(G) \leq n - \delta(G) + k$ where $\delta(G) \geq k$ and $n = |V(G)|$. This result begins to characterize graphs whose $k$-$L$-Grundy domination numbers equals the number of vertices in the graph, but it may not be a sufficient characterization. Thus, the following problem remains: 

\begin{problem}
Characterize graphs where $\gamma_\mathrm{gr}^{L,k}(G) = n$. 
\end{problem}

The study of grids above motivates the problem of how to bound the $k$-$L$-Grundy domination number for graph products. For example, if the requirement that $\delta \geq k$ is dropped, it can be shown that $\gamma_{gr}^{L,2}(P_n) = n$. Corollary~\ref{lgrid} states that $\gamma_{gr}^{L,2}(P_m \Box P_n) = mn-1$, which is less than  $\gamma_{gr}^{L,2}(P_m)\gamma_{gr}^{L,2}(P_n) = mn$. We believe that dropping the requirement $\delta \geq k$ might explain why $\gamma_{gr}^{L,2}(P_m \Box P_n)$ and $\gamma_{gr}^{L,2}(P_m)\gamma_{gr}^{L,2}(P_n)$ are not equal. Thus, we ask

\begin{problem}
Does $\gamma_{gr}^{L,k}(G \Box H) = \gamma_{gr}^{L,k}(G)\gamma_{gr}^{L,k}(H)$ when  $\delta \geq k$?
\end{problem}

This question is not even known in the case $k=1$.




\bibliographystyle{abbrv}
\bibliography{main}

\begin{thebibliography}{10}

\bibitem{AIMMINIMUMRANKSPECIALGRAPHSWORKGROUP20081628}
{AIM Minimum Rank – Special Graphs Work Group}.
\newblock Zero forcing sets and the minimum rank of graphs.
\newblock {\em Linear Algebra and its Applications}, 428(7):1628--1648, 2008.

\bibitem{amos2015upper}
D.~Amos, Y.~Caro, R.~Davila, and R.~Pepper.
\newblock Upper bounds on the k-forcing number of a graph.
\newblock {\em Discrete Applied Mathematics}, 181:1--10, 2015.

\bibitem{barioli2010zero}
F.~Barioli, W.~Barrett, S.~M. Fallat, H.~T. Hall, L.~Hogben, B.~Shader, P.~Van
  Den~Driessche, and H.~Van Der~Holst.
\newblock Zero forcing parameters and minimum rank problems.
\newblock {\em Linear Algebra and its Applications}, 433(2):401--411, 2010.

\bibitem{bell2021grundy}
K.~Bell, K.~Driscoll, E.~Krop, and K.~Wolff.
\newblock Grundy domination of forests and the strong product conjecture.
\newblock {\em Electronic Journal of Combinatorics}, 28(2):Paper 2.12 (18pp),
  2021.

\bibitem{benson2018zero}
K.~F. Benson, D.~Ferrero, M.~Flagg, V.~Furst, L.~Hogben, V.~Vasilevska, and
  B.~Wissman.
\newblock Zero forcing and power domination for graph products.
\newblock {\em Australasian Journal of Combinatorics}, 70(2):221, 2018.

\bibitem{bozeman2019restricted}
C.~Bozeman, B.~Brimkov, C.~Erickson, D.~Ferrero, M.~Flagg, and L.~Hogben.
\newblock Restricted power domination and zero forcing problems.
\newblock {\em Journal of Combinatorial Optimization}, 37(3):935--956, 2019.

\bibitem{brevsar2016dominating}
B.~Bre{\v{s}}ar, C.~Bujt{\'a}s, T.~Gologranc, S.~Klav{\v{z}}ar,
  G.~Ko{\v{s}}mrlj, B.~Patk{\'o}s, Z.~Tuza, and M.~Vizer.
\newblock Dominating sequences in grid-like and toroidal graphs.
\newblock {\em arXiv preprint arXiv:1607.00248}, 2016.

\bibitem{brevsar2017grundy}
B.~Bre{\v{s}}ar, C.~Bujt{\'a}s, T.~Gologranc, S.~Klav{\v{z}}ar,
  G.~Ko{\v{s}}mrlj, B.~Patk{\'o}s, Z.~Tuza, and M.~Vizer.
\newblock Grundy dominating sequences and zero forcing sets.
\newblock {\em Discrete Optimization}, 26:66--77, 2017.

\bibitem{brevsar2021graphs}
B.~Bre{\v{s}}ar and T.~Dravec.
\newblock Graphs with unique zero forcing sets and grundy dominating sets.
\newblock {\em arXiv preprint arXiv:2103.10172}, 2021.

\bibitem{brevsar2020grundy}
B.~Bre{\v{s}}ar, T.~Gologranc, M.~A. Henning, and T.~Kos.
\newblock On the l-grundy domination number of a graph.
\newblock {\em Filomat}, 34(10):3205--3215, 2020.

\bibitem{brevsar2014dominating}
B.~Bre{\v{s}}ar, T.~Gologranc, M.~Milani{\v{c}}, D.~F. Rall, and R.~Rizzi.
\newblock Dominating sequences in graphs.
\newblock {\em Discrete Mathematics}, 336:22--36, 2014.

\bibitem{brevsar2016total}
B.~Bre{\v{s}}ar, M.~A. Henning, and D.~F. Rall.
\newblock Total dominating sequences in graphs.
\newblock {\em Discrete Mathematics}, 339(6):1665--1676, 2016.

\bibitem{caro2014dynamic}
Y.~Caro and R.~Pepper.
\newblock Dynamic approach to k-forcing.
\newblock {\em arXiv preprint arXiv:1405.7573}, 2014.

\bibitem{edholm2012vertex}
C.~J. Edholm, L.~Hogben, M.~Huynh, J.~LaGrange, and D.~D. Row.
\newblock Vertex and edge spread of zero forcing number, maximum nullity, and
  minimum rank of a graph.
\newblock {\em Linear Algebra and its Applications}, 436(12):4352--4372, 2012.

\bibitem{ferrero2018relationship}
D.~Ferrero, L.~Hogben, F.~H. Kenter, and M.~Young.
\newblock The relationship between k-forcing and k-power domination.
\newblock {\em Discrete Mathematics}, 341(6):1789--1797, 2018.

\bibitem{herrman2022length}
R.~Herrman and S.~G. Smith.
\newblock On the length of l-grundy sequences.
\newblock {\em Discrete Optimization}, 45:100725, 2022.

\bibitem{herrman2022proof}
R.~Herrman and S.~G. Smith.
\newblock A proof of the grundy domination strong product conjecture.
\newblock {\em arXiv preprint arXiv:2212.04565}, 2022.

\bibitem{huang2010minimum}
L.-H. Huang, G.~J. Chang, and H.-G. Yeh.
\newblock On minimum rank and zero forcing sets of a graph.
\newblock {\em Linear Algebra and its Applications}, 432(11):2961--2973, 2010.

\bibitem{lin2019zero}
J.~C.-H. Lin.
\newblock Zero forcing number, grundy domination number, and their variants.
\newblock {\em Linear Algebra and its Applications}, 563:240--254, 2019.

\bibitem{nasini2020grundy}
G.~Nasini and P.~Torres.
\newblock Grundy dominating sequences on x-join product.
\newblock {\em Discrete Applied Mathematics}, 284:138--149, 2020.

\end{thebibliography}

\end{document}